\documentclass[11pt]{amsart}
\usepackage{amssymb,amsmath,amsfonts,latexsym}
\usepackage{bm}
\usepackage{mathtools}

\usepackage{array,graphics,color}
\numberwithin{equation}{section}

\newtheorem{dfn}{Definition}[section]
\newtheorem{thm}[dfn]{Theorem}
\newtheorem{lma}[dfn]{Lemma}

\newtheorem{ppsn}[dfn]{Proposition}
\newtheorem{crlre}[dfn]{Corollary}

\newtheorem{rmrk}[dfn]{Remark}

\newcommand{\D}{\mathbb{D}}
\newcommand{\C}{\mathbb{C}}		
\newcommand{\fcl}{\mathcal{F}}
\newcommand{\mcl}{\mathcal{M}}
\newcommand{\wcl}{\mathcal{W}}
\newcommand{\hilh}{\mathbb{H}}

\newcommand{\kcl}{\mathcal{K}}

\newcommand{\hdcr}{{H}^2_{\mathbb{C}^r}(\mathbb{D})}
\newcommand{\hdcrone}{H^2_{\mathbb{C}^{r+1}}(\mathbb{D})}
\newcommand{\hdcm}{{H}^2_{\mathbb{C}^m}(\mathbb{D})}
\newcommand{\hdcc}{{H}^2_{\mathbb{C}}(\mathbb{D})}

\DeclarePairedDelimiterX{\norm}[1]{\lVert}{\rVert}{#1}

\begin{document}
%\today

\title[ALMOST INVARIANT SUBSPACES OF THE SHIFT OPERATOR ON VECTOR-VALUED HARDY SPACES]{ALMOST INVARIANT SUBSPACES OF THE SHIFT OPERATOR ON VECTOR-VALUED HARDY SPACES}

%\dedicatory{Dedicated to Professor Kalyan Bidhan Sinha on the
%occasion of his 75th birthday}

\author[Chattopadhyay] {Arup Chattopadhyay}
\address{Department of Mathematics, Indian Institute of Technology Guwahati, Guwahati, 781039, India}
\email{arupchatt@iitg.ac.in, 2003arupchattopadhyay@gmail.com}

\author[Das]{Soma Das}
\address{Department of Mathematics, Indian Institute of Technology Guwahati, Guwahati, 781039, India}
\email{soma18@iitg.ac.in, dsoma994@gmail.com}

\author[Pradhan]{Chandan Pradhan}
\address{Department of Mathematics, Indian Institute of Technology Guwahati, Guwahati, 781039, India}
\email{chandan.math@iitg.ac.in, chandan.pradhan2108@gmail.com}



\begin{abstract}
	In this article, we characterize nearly invariant subspaces of finite defect for the backward shift operator acting on the vector-valued Hardy space which is a vectorial generalization of a result of Chalendar-Gallardo-Partington (C-G-P). Using this  characterization of nearly invariant subspace under the backward shift we completely describe the almost invariant subspaces for the shift and its adjoint acting on the vector valued Hardy space.
\end{abstract}
\maketitle

\section{Introduction}
%\renewcommand{\thefootnote}{\fnsymbol{footnote}} 
%\footnotetext{\textit{2010 Mathematics Subject Classification.} 47A13, 47A15, 47A80, 46E20, 47B38, 47B32,  30H10.}
%\footnotetext{\emph{Key words and phrases.} Vector valued Hardy space, nearly invariant subspaces, almost invariant subspaces, shift operator, Beurlings theorem, half space, multiplier operator.}     
%\renewcommand{\thefootnote}{\arabic{footnote}} 
\hspace*{3mm} In 1988, Hitt \cite{HIT} first introduces the notion of nearly invariant subspaces under the backward shift operator acting on the scalar-valued Hardy space which he used as a tool for classifying the simply shift-invariant subspaces of the Hardy space of an annulus. In his paper he rather called it as  \textquotedblleft weakly invariant subspace under the backward shift \textquotedblright. Later  Sarason \cite{SR} further investigated these spaces and modified Hitt's algorithm for scalar-valued Hardy space to study the kernels of Toeplitz operators. In 2010,  Chalendar-Chevrot-Partington (C-C-P) \cite{CCP} gives a complete characterization of nearly invariant subspaces  under the backward shift operator acting on the vector-valued Hardy space,  providing a vectorial generalization of a result of Hitt. Recently Chalendar-Gallardo-Partington (C-G-P) \cite{CGP} introduce the notion of nearly invariant subspace of finite defect for the backward shift operator acting on the scalar valued Hardy space as a generalization of nearly invariant subspaces and provides a complete characterization of these spaces in terms of backward shift invariant subspaces. Using this characterization they also described the almost-invariant subspaces for the shift and its adjoint acting on the scalar valued Hardy space. In this connection we should mention that the relation between nearly invariant subspaces under the backward shift and the kernel of Toeplitz operators has been discussed in \cite{NC}.

In this paper we further study nearly invariant subspaces of finite defect under the backward shift operator acting on the  vector valued Hardy space and provides a vectorial generalization of C-G-P algorithm. As a consequences we completely characterize nearly invariant subspaces  of finite defect under the backward shift in terms of backward shift invariant subspaces. Furthermore,  using the characterization of nearly invariant subspace under the backward shift we completely describe the almost invariant subspaces for the shift and its adjoint acting on the vector valued Hardy space. Moreover at the end we also provide a connection between the orthocomplement of a nearly invariant subspaces of finite defect under the backward shift and the shift invariant subspaces on the vector valued Hardy space.  

The $\mathbb{C}^m$- valued Hardy space \cite{JP} over the unit disc $\mathbb{D}$ is denoted by $\hdcm$ and defined by $$\hdcm:=\Big\{F(z)=\sum_{n\geq 0} A_nz^n:~\|F\|^2= \sum_{n\geq 0}~\norm{A_n}_{\mathbb{C}^m}^2<\infty,~A_n\in\mathbb{C}^m\Big\}.$$ We can also view the above Hilbert space as the direct sum of $m$-copies of $H^2_{\mathbb{C}}(\D)$ or  sometimes it is useful to see the above space as a tensor product of two Hilbert spaces $H^2_{\mathbb{C}}(\D)$ and $\mathbb{C}^m$, that is, $$H^2_{\mathbb{C}^m}(\D)  \equiv   \underbrace{H^2_{\mathbb{C}}(\D)\oplus \cdots\oplus H^2_{\mathbb{C}}(\D)}_{m}\equiv H^2_{\mathbb{C}}(\D)\otimes \mathbb{C}^m.$$
Let $S$ denote the forward shift operator  (multiplication by the independent variable)
acting on $\hdcm$, that is, $SF(z)=zF(z)$,~$z\in \D$. The adjoint of $S$ is denoted by $S^*$ and defined in $\hdcm$ as the operator 
$$S^*(F)(z)=\dfrac{F(z)-F(0)}{z},~~ F\in\hdcm$$ which is known as backward shift operator.
The Banach space of all $\mathcal{L}(\mathbb{C}^r,\mathbb{C}^m)$ (set of all bounded linear operators from $\mathbb{C}^r$ to $\mathbb{C}^m$)- valued  bounded analytic functions on $\mathbb{D}$ is denoted by 
$H^{\infty}_{\mathcal{L}(\mathbb{C}^r,\mathbb{C}^m)}(\D)$. Each $\Theta\in H^{\infty}_{\mathcal{L}(\mathbb{C}^r,\mathbb{C}^m)}(\D)$ induces a bounded linear map $T_{\Theta}\in H^{\infty}_{\mathcal{L}(\mathbb{C}^r,\mathbb{C}^m)}(\D)$ defined by
$$T_{\Theta}F(z)=\Theta(z)F(z).~~(F\in H^{2}_{\mathbb{C}^r}(\D))$$
The elements of $H^{\infty}_{\mathcal{L}(\mathbb{C}^r,\mathbb{C}^m)}(\D)$  are called the \emph{multipliers} and are determined by
$$\Theta\in H^{\infty}_{\mathcal{L}(\mathbb{C}^r,\mathbb{C}^m)}(\D)\textit{~if~and~only~if~} ST_{\Theta}=T_{\Theta}S,$$
where the shift $S$ on the left hand side and the right hand side act on $ H ^{2}_{\mathbb{C}^m}(\D)$ and $ H ^{2}_{\mathbb{C}^r}(\D)$ respectively. A multiplier $\Theta \in  H ^{\infty}_{\mathcal{L}(\mathbb{C}^r,\mathbb{C}^m)}(\D)$ is said to be \emph{inner} if $M_{\Theta}$ is an isometry, or equivalently, 
$\Theta(z)\in \mathcal{L}(\mathbb{C}^r,\mathbb{C}^m)$ is an isometry almost everywhere with respect to the Lebesgue measure on $\mathbb{T}$ (unit circle). Inner multipliers are among the most important tools for classifying invariant subspaces of reproducing kernel Hilbert spaces. For instance: 

\begin{thm}\label{a1}
	(Beurling-Lax-Halmos \cite{NF}) 
	A non-zero closed subspace $\mathcal{M}\subseteq  H ^{2}_{\mathbb{C}^m}(\D)$ is shift
	invariant if and only if there exists an inner multiplier $\Theta \in H ^{\infty}_{\mathcal{L}(\mathbb{C}^r,\mathbb{C}^m)}(\D)$ such that
	$$\mathcal{M} = \Theta  H ^{2}_{\mathbb{C}^r}(\D),$$
	for some $r$ ($1\leq r\leq m$).  
\end{thm}
\noindent Consequently, the space $\mathcal{M}^{\perp}$ of 
$H^2_{\mathbb{C}^m}(\D)$ which is invariant under $S^*$ (backward shift) can be represented as 
$$\mathcal{K}_{\Theta}: = \mathcal{M}^{\perp} =  H^2_{\mathbb{C}^m}(\D) \ominus \Theta  H^2_{\mathbb{C}^r}(\D),$$
which also known as model spaces \cite{EMV1,EMV2,NB}.
\begin{dfn}
	A closed subspace $\mathcal{M}$ of $\hdcm$ is said to be \emph{almost-invariant} for S if there exists a finite dimensional subspace $\fcl$ of $\hdcm$ such that $$S(\mathcal{M})\subseteq \mathcal{M}+\fcl.$$
\end{dfn}
\noindent The space $\fcl$ is called the defect space and the smallest possible dimension of $\fcl$ is called defect of the space $\mathcal{M}$. Moreover, a space $\mathcal{M}$ is called a \emph{half space} if $\mathcal{M}$ has infinite dimension and infinite co-dimension. The study of almost-invariant half-spaces of any bonded linear operators $T$ acting on complex Banach spaces was initiated in 2009 due to Androulakis, Popov, Tcaciuc and Troitsky \cite{APTT} and later it further studied by Popov, Tcaciuc, Sirotkin and Wallis \cite{PT, SW, T} to investigate the structure of almost-invariant half-spaces in more general setting. In this connection it is easy to observe that every subspace which is not a half-space is clearly almost-invariant under any operator. A well-known result due to Beurling \cite{B} states that if $\mathcal{M}$ is a $S$-invariant subspace of $ H^2_{\mathbb{C}}(\D)$,  then $\mathcal{M}$ can be represented as 
$$\mathcal{M}=\theta  H^2_{\mathbb{C}}(\D), $$
where $\theta\in  H^\infty_{\mathbb{C}}(\D)$ is an inner function (that is, $\theta$ is a bounded holomorphic function on $\mathbb{D}$ and $|\theta|=1$ a.e. on $\mathbb{T}$). In this regard it is not difficult to conclude that the shift invariant subspace $\mathcal{M}$ of $ H ^{2}_{\mathbb{C}}(\D)$ is a half-space if and only if $\mathcal{M}=\theta  H^2_{\mathbb{C}}(\D)$ with $\theta$ is not rational (that means $\theta$ is not a product of finitely many Blaschke factor) \cite{RSN}.

The purpose of this paper is to characterize almost invariant subspaces for the shift and its adjoint acting on the vector valued Hardy space in terms of invariant subspaces for the adjoint $S^*$ with finite defect. To achieve our goal we give a connection between nearly invariant subspaces with finite defect and invariant subspaces for $S^*$ in the vector valued Hardy space (see. Theorem~\ref{ab}).  

%Firstly, we will prove that any nearly invariant subspace $\mcl$ with finite defect for $S^*$ in $\hdcm$ can be described via invariant subspace $\kcl$ of $S^*$ in vector valued Hardy space which is a known subspace due to Beurling-Lax-Halmos theorem (Theorem \ref{a1}). 

%Before that we need to know about nearly invariant subspace for $S^*$ in vector valued Hardy space.

\begin{dfn}
	A closed subspace $\mathcal{M}$ of $\hdcm$ is said to be nearly invariant for $S^*$ if every element  $F\in \mathcal{M}$ with $F(0)=0$ satisfies $S^*F\in \mathcal{M}$.
\end{dfn}

The paper is organized as follows: In Section 2 we give a connection between nearly invariant subspaces for $S^*$ and almost-invariant subspaces for $S$ in vector valued Hardy spaces. In other words we generalize some results of (\cite{CGP} ,~Section 2) in the vector valued setting. Section 3 deals with the main result of this paper. At the end  we obtain a connection between the orthocomplement of a class of nearly invariant subspaces of finite defect under the backward shift and the shift invariant subspaces on the vector valued Hardy space.  

%\vspace{4in}

%More details about basic Hardy space theory can be found in \cite{RSN},\cite{EMV1},\cite{NUEM}.\vspace*{2mm}

%\hspace*{3mm} Let $V\in\mathcal{L}(\hdcm)$, a closed subspace $\mcl$ of $\hdcm$ is called $V$-invariant subspace if $V(\mcl)\subseteq\mcl$. Beurling-Lax-Halmos theorem \cite{PJ} says that every $S$ invariant subspace $\mcl\subseteq\hdcm$ have the form $\mcl=\Theta\hdcr$, where $0\leq r\leq m$ and $\Theta\in H ^{\infty}(\D,\mathcal{L}(\mathbb{C}^r,\mathbb{C}^m))$ is inner multiplier \cite{EMV1},\cite{EMV2},\cite{NUEM}. As in scalar case, $\mathbf{K}_\Theta=\hdcm\ominus\Theta\hdcr$ is the general form of $S^*$-invariant subspace of $\hdcm$.\vspace*{1.5mm}

%\hspace*{3mm} A closed subspace $\mathcal{N}$ of $\hdcm$ is said to be almost invariant for S if there exist a finite dimensional subspace $\fcl$ of $\hdcm$ such that $S(\mathcal{N})\subseteq \mathcal{N}+\fcl$. The smallest possible dimension of $\fcl$ is called defect of $\mathcal{N}$ and $\fcl$ is called the defect space.\vspace*{1.5mm}\\
%A subspace $\mcl\subseteq\hdcc$ is called a nearly $S^*$-invariant subspace if $S^*(f)\in\mcl$ whenever $f(0)=0$.\vspace*{.5mm}

\section{Preliminary Results}
%Any invariant subspace $\mathcal{M}$ in $\hdcm$ is invariant under S if and only if $\mathcal{M}= \Theta \hdcr$ where $\Theta \in \mathbf{H}^{\infty}(\mathbb{D},\mathcal{L}(\mathbb{C}^r, \mathbb{C}^m))$ is an inner multiplier and $1\leq r \leq m$. Therefore the invariant subspaces for $S^*$ is $\mathcal{K}_{\Theta}:=(\Theta \hdcr)^\perp$ in $\hdcm$, these are known as model spaces.\\ 
%A subspace is said to be a half-space if it is of infinite dimension and it's codimension is also infinite. Now, if we consider 
%\begin{equation}
%  \Theta=
%   \begin{pmatrix}  
%   \theta _1 & 0 & 0 & . & . & . & 0 \\[1pt]
%   0 & \theta_2 & 0 & . & . & . & 0\\
% . & . & . & . & . & . & .\\
% . & . & . & . & . & . & .\\
% . & . & . & . & . & . & .\\
% 0 & 0 & 0& . & . & . & \theta_m \\
%   \end{pmatrix}_{m\times m} 
%\end{equation}
% where each $\theta _i (1\leq i \leq m)$ is an inner function of $\hdcc$ then $\Theta $ is an inner multiplier in $\hdcm$. Therefore, if one of $\theta _i$ is not rational inner function i.e., it is not a finite Blaschke product then $\Theta \hdcm$ is an shift invariant half-space. 

The space $\hdcm$ can also be defined as the collection of all $\mathbb{C}^m$-valued analytic functions $F$ on $\mathbb{D}$ such that 
$$\norm{F}= \Big[~\sup_{0\leq r<1} \frac{1}{2\pi} \int_0^{2\pi} \vert F(re^{i\theta})\vert^2~ d\theta~\Big]^{\frac{1}{2}}<\infty.$$ 
Moreover the nontangential boundary limit (or radial limit) $$F(e^{i\theta}):= \lim\limits_{r\rightarrow 1-}F(re^{i\theta})$$ exists almost everywhere on $\mathbb{T}$ (for more details see \cite{NB}, I.3.11). Therefore $\hdcm$ can be embedded isomertically as a closed subspace of
$L^2(\mathbb{T},\mathbb{C}^m)$ by identifying $\hdcm$ through the nontangential boundary limits of the $\hdcm$ functions. Furthermore $L^2(\mathbb{T},\mathbb{C}^m)$ can be decomposed in the following way 
$$L^2(\mathbb{T},\mathbb{C}^m)=\hdcm \oplus \overline{ H _0^2},$$ where $\overline{ H _0^2}=\{F\in L^2(\mathbb{T},\mathbb{C}^m):\overline{F}\in \hdcm ~and~ F(0)=0 \}$. 
\begin{dfn}
	A closed subspace $\mathcal{M}$ of $ H ^2_{\mathbb{C}}(\mathbb{D})$ is said to be nearly invariant for the backward shift $S^*$ if every element  $f\in \mathcal{M}$ with $f(0)=0$ satisfies $S^*f\in \mathcal{M}$.
\end{dfn}
As in the introduction we already noticed that nearly $S^*$-invariant subspaces of $ H ^2_{\mathbb{C}}(\D)$ were introduced and characterized by Hitt \cite{HIT} and Sarason \cite{SR}. The vectorial generalization of Hitt's  and Sarason's result was due to  Chalendar-Chevrot-Partington (C-C-P) \cite{CCP} which says the following: Every non trivial nearly $S^*$-invariant subspace $\mathcal{M}$ of $\hdcm$ has the form $\mathcal{M}=F_0\mathcal{K},$ where $F_0$ is the $m\times r~(1\leq r \leq m)$ matrix whose columns are $\{W_1,W_2,...,W_r\}$ which forms an  orthonormal basis of $\wcl:=\mcl\ominus(\mcl\cap z\hdcm)$ and $\mathcal{K}$ is a $S^*$- invariant subspace of $\hdcr$. Therefore by Beurling-Lax-Halmos theorem there exists an inner multiplier $\Theta \in  H ^{\infty}_{ \mathcal{L}(\mathbb{C}^{r^\prime},\mathbb{C}^r)}(\mathbb{D})$  for some $r^\prime$ ($\leq r$) such that $$\mathcal{K}=\mathcal{K}_{\Theta}:= \hdcr \ominus\Theta  H ^2_{\mathbb{C}^{r^\prime}}(\D)$$ with an extra property that $\Theta (0)=0$. The operator $$T_{F_0}:\hdcr \rightarrow \hdcm$$  $$G\longmapsto P(F_0G),$$ where $P$ is the Fourier projection of the $L^1(\mathbb{T},\mathbb{C}^m)$ function $F_0G$ on $ H ^2_{\mathbb{C}^m}(\mathbb{D})$, as in the scalar case it is an isometry from $\mathcal{K}_{\Theta}~ onto~ \mathcal{M}$.\\
As discussed earlier throughout this section we are going to provide vectorial generalization of  results given in (\cite{CGP} ,~Section 2).
Now we are in a position to prove our first result   which produces a connection between nearly invariant subspaces  for $S^*$ and almost-invariant subspaces  for $S$ in the vector valued case.
\begin{ppsn}\label{pr1}
	Let $F_0$ and $\mathcal{K}_{\Theta}$ be as above. Then the nearly $S^*$- invariant subspace  $\mathcal{M}=F_0 \mathcal{K}_{\Theta}$ is an almost invariant for S with defect $r^\prime$. In particular, if the inner multiplier $\Theta \in  H ^{\infty}_{\mathcal{L}(\mathbb{C}^r)}(\mathbb{D})$ is of the form:
	\begin{equation}\label{1234}
	\Theta=
	\begin{pmatrix}  
	\theta _1 & 0 & 0 & . & . & . & 0 \\[1pt]
	0 & \theta_2 & 0 & . & . & . & 0\\
	. & . & . & . & . & . & .\\
	. & . & . & . & . & . & .\\
	. & . & . & . & . & . & .\\
	0 & 0 & 0& . & . & . & \theta_r \\
	\end{pmatrix}_{r\times r}, 
	\end{equation}
	where  $\{\theta _1,\theta_2,\ldots,\theta_r\}$ is a collection of inner functions of $\hdcc$ with at least one $\theta _i$ (say) is not rational, then $\mathcal{M}=F_0 \mathcal{K}_{\Theta}$ is an almost-invariant half space with defect $r$.
\end{ppsn}
\begin{proof}
	Since $\Theta \in  H ^{\infty}_
	{\mathcal{L}(\mathbb{C}^{r\prime},\mathbb{C}^r)}(\D)$ is an inner multiplier, then the map  $$T_{\Theta}: H ^2_{\mathbb{C}^{r\prime}}(\D)\rightarrow \hdcr$$ is an isometry. Let $\{e_i\}_{i=1}^{r\prime}$ be an orthonormal basis of $\mathbb{C}^{r\prime}$. Now consider $\widetilde {\Theta_i} = \Theta e_i \in \hdcr,~ \text{for}~ i=1,2,...,r{\prime}.$ Note that  $T_{\Theta}$ is an isometry implies $\{\widetilde {\Theta_i}\}_{i=1}^{r\prime}$ is linearly independent in $\hdcr$. Moreover $$(\mathcal{K}_{\Theta}+~span \{\widetilde{\Theta_i}\}_{i=1}^{r\prime})^\perp=\Theta  H ^2_{\mathbb{C}^{r\prime}}(\D)\cap (span \{\widetilde{\Theta_i}\}_{i=1}^{r\prime})^\perp=z\Theta H ^2_{\mathbb{C}^{r\prime}}(\D).$$ On the other hand for $G\in  H ^2_{\mathbb{C}^{r\prime}}(\D)$ and $F\in \mathcal{K}_{\Theta}$ we have $\langle z\Theta G, zF\rangle =0$ and hence
	$z\Theta H ^2_{\mathbb{C}^{r\prime}}(\D) \subseteq (z\mathcal{K}_{\Theta})^\perp$. Thus $S\mathcal{K}_{\Theta} \subseteq \mathcal{K}_{\Theta}+span \{\widetilde{\Theta_i}\}_{i=1}^{r\prime} $. Since $T_{F_0}$ is an isometry from $\mathcal{K}_{\Theta}$ onto $\mathcal{M}$ we have  $S\mathcal{M} \subseteq \mathcal{M} + span\{F_0 \widetilde{\Theta _i}\}_{i=1}^{r\prime}$.
	This proves that $\mathcal{M}$ is almost invariant under $S$ with defect $r^\prime$.
	
	For the second part we assume that  $\theta _j$ is not rational for some $j\in \{1,\ldots, r\}$  which immediately implies that $\theta_j \hdcc$ is a half space. Note that since $\Theta$ is of the form \eqref{1234}, then $\Theta  H ^2_{\mathbb{C}^{r}}(\D)$ is again a half space which concludes that $\mathcal{M}$ is also a half space. This concludes the proof.
\end{proof}
The following two lemmas are very useful to conclude  that the orthocomplement of some nearly invariant subspaces for $S^*$ are also almost invariant for $S$ of some finite defect in $\hdcm$. 
\begin{lma}\label{lm1}
	Let $\Psi \in  H ^{\infty}_{\mathcal{L}(\C ^m)}(\D)$ be an inner multiplier of the form 
	\begin{equation}\label{12345}
	\Psi=
	\begin{pmatrix}  
	\psi _1 & 0 & 0 & . & . & . & 0 \\[1pt]
	0 & \psi_2 & 0 & . & . & . & 0\\
	. & . & . & . & . & . & .\\
	. & . & . & . & . & . & .\\
	. & . & . & . & . & . & .\\
	0 & 0 & 0& . & . & . & \psi_m \\
	\end{pmatrix}_{m\times m.},
	\end{equation}
	where  $\{\psi _1,\psi_2,\ldots,\psi_m\}$ is a collection of inner functions of $\hdcc$. Then $(\Psi \mathcal{K}_{\Theta})^\perp = \Psi \Theta  H ^2 _{\mathbb{C}^r}(\D) \oplus \mathcal{K}_\Psi$ for any inner multiplier $\Theta \in  H ^{\infty}_{\mathcal{L}(\C ^r,\C ^m)}(\D)$ with $r \leq m$.
\end{lma}
\begin{proof}
	Let $F\in \hdcm$. Then for all $K\in \mathcal{K}_{\Theta}$ we have 
	%\begin{align*}
	$$\langle F,\Psi {K} \rangle _{\hdcm} 
	= \langle F,\Psi K \rangle _{L^2(\mathbb{T},\mathbb{C}^m)}
	=\langle T^*_\Psi F,K \rangle _{L^2(\mathbb{T},\mathbb{C}^m)}.$$
	Therefore $T^*_\Psi F \in \Theta \hdcr\oplus \overline{ H ^2_0}
	\textit{~if~and~only~if~} F \in \Psi \Theta \hdcr\oplus \mathcal{K}_{\Psi}$, where $ \mathcal{K}_{\Psi}=(\Psi\hdcm)^\perp = \hdcm\cap \Psi \overline{ H ^2_0}$.
	%\end{align*}
	This completes the proof.
\end{proof}

\begin{lma}\label{lm2}
	Let $\Psi \in  H ^{\infty}_{\mathcal{L}(\C ^m)}(\D)$ be as in the statement of lemma \ref{lm1} with an extra assumption that $\psi _i(0)\neq 0$ for each $i\in \{1,2,\ldots,m\}$. Then  $\Psi \mathcal{K}_{\Theta}$ is nearly $S^*$ invariant for any inner multiplier $\Theta \in  H ^{\infty}_{\mathcal{L}(\C^r,\C^m)}(\D) $  with $r\leq m.$
\end{lma}
\begin{proof}
	Let $F\in \Psi \mathcal{K}_{\Theta}$ be such that $F(0)=0$. Then $$F=\Psi K,~ \text{where} ~K\in  \mathcal{K}_{\Theta}~ \text{and} ~ F(0)= \Psi (0)K(0)=0.$$ Since each $\psi _i(0)\neq 0$ for each $i\in \{1,2,\ldots,m\}$, then from the above we conclude that $K(0)=0$. Thus $$S^*F(z) = \dfrac{F(z)-F(0)}{z}=\dfrac{\psi(z)K(z)}{z}=\Psi (z)S^*K(z), \forall z\in D.$$ Since $\mathcal{K}_{\Theta}$ is $S^*$ invariant, then $S^*F \in \Psi \mathcal{K}_{\Theta}$. This completes the proof.
\end{proof}
Combining lemma \ref{lm1} and lemma \ref{lm2} we have the following result.
\begin{ppsn}\label{pr2}
	Let $\Psi \in  H ^{\infty}_{\mathcal{L}(\C ^m)}(\D)$ be as in the statement of lemma \ref{lm1} and let $\Theta \in  H ^{\infty}_{\mathcal{L}(\C^r,\C^m)}(\D) ~for~ r\leq m$. Then $(\Psi \mathcal{K}_{\Theta})^\perp$ is an almost invariant subspace for $S$ in $\hdcm$ with defect $m$. In particular if 
	\begin{equation}\label{123456}
	\Theta=
	\begin{pmatrix}  
	\theta _1 & 0 & 0 & . & . & . & 0 \\[1pt]
	0 & \theta_2 & 0 & . & . & . & 0\\
	. & . & . & . & . & . & .\\
	. & . & . & . & . & . & .\\
	. & . & . & . & . & . & .\\
	0 & 0 & 0& . & . & . & \theta_m \\
	\end{pmatrix}_{m\times m,}, 
	\end{equation}
	where  $\{\theta _1,\theta_2,\ldots,\theta_m\}$ is a collection of inner functions of $\hdcc$ with at least one $\theta _i$ (say) is not rational, then $(\Psi \mathcal{K}_{\Theta})^\perp$ is an almost invariant half space of defect $m$.	
\end{ppsn}
\begin{proof}
	By repeating the similar kind of argument as in  the proof of proposition \ref{pr1} we conclude that $$S\mathcal{K}_{\Psi}\subset \mathcal{K}_{\Psi}+span\{\widetilde{\Psi_i}\}_{i=1}^m,$$ where $\widetilde{\Psi_i}=\Psi e_i$ and $\{e_1,e_2,\ldots,e_m\}$ is an orthonormal basis of $\mathbb{C}^m$. On the other hand by lemma \ref{lm1} we have $(\Psi \mathcal{K}_{\Theta})^\perp= \Psi \Theta \hdcm \oplus \mathcal{K}_\Psi.$ Thus by combining this two results we have the following :$$S(\Psi \mathcal{K}_{\Theta})^\perp \subset (\Psi \mathcal{K}_{\Theta})^\perp+ span\{\widetilde{\Psi_i}\}_{i=1}^m.$$
	Since by hypothesis $\Theta$ is of the form \eqref{123456}, then the dimensions of both $(\Psi \mathcal{K}_{\Theta})^\perp$ and $\Psi \mathcal{K}_{\Theta}$ are infinite and hence it is a half space. This completes the proof.
\end{proof}

In proposition \ref{pr1} we have seen that every nearly invariant subspace for $S^*$ is an almost invariant subspace for $S$. The next proposition says that the converse of this result is not true that is, there exists an almost invariant half space for S which is not nearly $S^*$- invariant. 
\begin{ppsn}
	Let $\Theta \in  H ^{\infty}_{\mathcal{L}(\C ^m)}(\D)$ be as in (\ref{123456})
	with an extra assumption that  $\Theta(0)=0$. Then  $(\Theta \mathcal{K}_{\Theta})^\perp$ is an almost invariant half space for S of defect m but not nearly invariant for $S^*$. 	
\end{ppsn}
\begin{proof}
	Since $\Theta$ is of the form \eqref{123456}, then by proposition \ref{pr2} we conclude that $(\Theta \mathcal{K}_{\Theta})^\perp$ is an almost invariant for $S$. 
	Note that $\Theta$ is an inner multiplier of the form \eqref{123456} with at least one $\theta_j$ (say) is not rational. Now $\Theta (0)=0$ implies that $\theta _i(0)=0 ~\text{for all}~ i\in \{1,2,...,m\}.$ Let $F=\Theta ^2e_j$, where $\{e_1,e_2,\ldots,e_m\}$ is an orthonormal basis of $\mathbb{C}^m$. Then $F\in \hdcm$ and $F(0)=0.$ Note that  for any $K\in \mathcal{K}_\Theta$, $\langle F,\Theta K\rangle = \langle \Theta ^2 e_j,\Theta K\rangle =0$ and hence $F\in (\Theta \mathcal{K}_{\Theta})^\perp$. On the contrary, let us assume that $(\Theta \mathcal{K}_{\Theta})^\perp$ is nearly invariant for $S^*$.  Therefore $S^*F \in (\Theta \mathcal{K}_{\Theta})^\perp = \Theta ^2\hdcm \oplus \mathcal{K}_\Theta$ (by lemma \ref{lm1}) and hence 
	\begin{equation}\label{eqp1}
	S^*F(z)= \frac{F(z)}{z}= \Theta ^2(z)H(z)+ K(z),
	\end{equation}
	for some $H\in\hdcm$ and $K \in \mathcal{K}_{\Theta}$. On the other hand $\Theta (0)=0$ implies that there exists an another inner multiplier $\Theta _1\in  H ^{\infty}_{\mathcal{L}(\mathbb{C}^m,\mathbb{C}^m)}(\D,)$ such that $\Theta (z)=z\Theta _1(z),~ \forall z\in \D $ and hence 
	\begin{equation}\label{eqp2}
	S^*F(z)=\frac{F(z)}{z}=\frac{\Theta ^2(z)e_j}{z}=\Theta (z)\Theta_1(z)e_j \in \Theta\hdcm.
	\end{equation}
	Combining \eqref{eqp1} and \eqref{eqp2} we conclude that $K \in \mathcal{K}_{\Theta}\cap \Theta\hdcm$ and therefore $K(z)=0.$ This implies that $H(z)=\frac{1}{z}\otimes e_j \in \hdcm$ which is not the case. This completes the proof.

\end{proof}

\section{Classification of almost invariant subspaces}
The main aim of this section is to describe completely the almost invariant subspaces for the shift and its adjoint acting on the vector valued Hardy space $\hdcm$. At first we  begin with the definition of nearly invariant subspace for $S^*$ with finite defect on the vector valued Hardy space which already introduced  in \cite{CGP} for the scalar valued Hardy space $\hdcc$.

\begin{dfn}
	A closed subspace $\mcl \subset \hdcm$ is said to be nearly $S^*$-invariant with defect p if and only if there is an $p$-dimensional subspace $\fcl \subset \hdcm$ (which may be taken to be orthogonal to $\mcl$ ) such that if $F \in M,F(0)=0$ then $S^*F$ $\in M\oplus \fcl$. We say that $M$ is $S^*$ almost invariant with defect $p$ if and only if $S^*\mcl \subset \mcl \oplus \fcl $; where dim $\fcl =p$.
\end{dfn}
Characterization of the subspaces of a vector-valued Hardy space that are nearly $S^*$ invariant was due to Challender, Chervot, Partington \cite{CCP} which provides a vectorial generalization of a result of Hitt \cite{HIT}. Recently Challeder, Gallardo, Partington (C-G-P) gives a complete characterization of nearly $S^*$ invariant  subspaces with finite defect in the scalar valued Hardy space $\hdcc$ \cite{CGP}. Here our principle aim is to provide a complete characterization of nearly $S^*$ invariant subspaces with finite defect in the vector valued Hardy space $\hdcm$. Before going to the main result of this section we first start with the following two lemmas. Note that the first lemma already proved in \cite{CCP} but for reader's convenience we are providing a proof herewith.  
\begin{lma}
	Let $\mcl$ be a closed subspace of $\hdcm $ such that all functions in $\mcl$ do not vanish at $0$. Then
	$$ 1\leq\dim{(\mcl\ominus(\mcl\cap z\hdcm))}\leq m.$$
\end{lma}
\begin{proof}
	Note that by hypothesis the space $\wcl:=\mcl\ominus(\mcl\cap z\hdcm)$ is non-trivial. Let $\dim\wcl=r$. For $i\in\{1,2,...,m\}$, let $F_i=P_\mcl(k_0\otimes e_i)$, where $P_\mcl$ is the orthogonal projection of $\hdcm $ onto $\mcl$ and $k_0$ is the reproducing kernel at 0 in $\hdcc$. Then $F_i\in\mcl$ and $\langle F_i, G\rangle =0$ for all $G\in \mcl\cap z\hdcm$ which implies that $G(0)=0$. This shows that  $\{F_i\}_{i=1}^m$  generates the space $\wcl$ and hence $\dim {\wcl}\leq m$. This completes the proof.
\end{proof}
Let $\hilh$ be a complex separable Hilbert-space and define  $$C_{\cdot 0}:=\Big\{T\in\mathcal{L}(\hilh, \hilh
): {T^*}^nh\to 0~\text{as~}  n\to \infty ~\text{for~all~} h\in\hilh\Big\}.$$
\begin{lma}(Lemma 3.3 in \cite{TIM})\label{abc}
	Suppose $T\in C_{\cdot 0}$ and $\dim~{\mathcal{D}_T} \big(=\overline{Ran} (I-T^*T)^{\frac{1}{2}}\big)<\infty$. Let $\fcl\subset\hilh$ be a closed subspace of finite codimension. Then $TP_\fcl\in C_{\cdot 0}$, where $P_\fcl$ denotes the orthogonal projection onto $\fcl$.
\end{lma}
Now we are in a position to state and prove the main theorem in this section.
\begin{thm}
	Let $\mcl$ be a closed subspace that is nearly $S^*$-invariant with defect 1 in $\hdcm$ and let $E_1\in \hdcm$ be such that $\fcl=\langle E_1 \rangle$ (subspace spanned by the vector $E_1$) is the defect space with $\norm{E_1}=1$. Let $\{W_1,W_2,\ldots,W_r\}$ be an orthonormal basis of $\wcl:=\mcl\ominus(\mcl\cap z\hdcm)$ and let $F_0$ be the $m\times r$ matrix whose columns are $W_1,W_2,\ldots,W_r$. Then\vspace*{0.1in}\\
	(i) in the case where there are functions in $\mcl$ that do not vanish at $0$,
	\begin{equation}
	\begin{split}
	&\mcl=\Big\{F\in \hdcm : F=F_0K_0+zk_1E_1 \quad \text{and}\\
	&\hspace{2.5in} (K_0,k_1)\in \kcl\subset \hdcr\times\hdcc \Big\},\label{M}
	\end{split}
	\end{equation}
	where\begin{align*}
	&\kcl=\Big\{ (K_0,k_1)\in\hdcrone:\exists~ F\in \mcl \text{~such~that~}\\
	&\hspace{3.5in} F=F_0K_0+zk_1E_1\Big\}
	\end{align*}
	is a closed $S^*\oplus\cdots \oplus S^*$-invariant subspace of the vector valued Hardy space $\hdcrone$, and $\norm{F}^2=\norm{K_0}^2+\norm{k_1}^2$.\vspace*{0.1in}\\
	(ii)  In the case where all functions in $\mcl $ vanish at $0$.
	$$\mcl= \{F : F(z)= zk_1(z)E_1(z) : k_1\in \kcl\}, $$ 
	where $\kcl$ is now a closed $S^*$- invariant subspace of the  Hardy space $\hdcc$ and $\norm{F}^2=\norm{k_1}^2 .$\\
	
	\noindent Conversely, if a closed subspace $\mcl$ of the vector valued Hardy space $\hdcm$ has a representation like $(i)$ or $(ii)$ as above, then it is a nearly $S^*$-invariant subspace of defect 1. 
\end{thm}
\begin{proof}			
	$(i)$ By hypothesis $\mcl\nsubseteq z\hdcm$.
	Let $P_\wcl$ denote the orthogonal projection of $\mcl$ onto $\wcl$. If $F\in\mcl$, then $P_\wcl(F)$ can be written as  $$P_\wcl(F)(z)=a_{0,1}W_1(z)+\cdots+a_{0,r}W_r(z),$$ and hence for each $z \in \D$, $F(z)=P_\wcl(F)(z)+F_1(z)$, where $F_1\in\mcl\cap\wcl^\perp$. Since $\{W_1,W_2,\ldots,W_r\}$ forms an orthonormal basis of $\wcl$, then we have the following norm identity: 
	\begin{align*}
	\norm{F}^2&=|a_{0,1}|^2+|a_{0,2}|^2+\cdots+|a_{0,r}|^2+\norm{F_1}^2\\	&=\norm{A_0}^2 +\norm{F_1}^2,\text{~where~}A_0=(a_{0,1},a_{0,2},\ldots,a_{0,r})^T.	\end{align*} Note that  $F_1\in\mcl\cap\wcl^\perp$ and hence $F_1(0)=0$. On the other hand since $\mcl$ is a nearly $S^*$-invariant subspace with defect 1, then $S^*F_1=G_1+\beta_1E_1$, where $G_1\in\mcl$ and $\beta_1\in\C$. Therefore $F_1=S(G_1+\beta_1E_1)$ because $F_1(0)=0$ and $SS^*F_1=F_1$. Thus for $z\in \mathbb{D}$ we have 
	\begin{equation}\label{app4}
	F(z)=F_0(z) A_0+zG_1(z) +\beta_1zE_1(z) \quad \text{and} \quad \|F\|^2 = \norm{A_0}^2 + \norm{G_1}^2 + |\beta_1|^2.
	\end{equation}                                                                                                                                                                                                                                                        
	Now we repeat the above process starting with $G_1$. Then $G_1=F_0A_1+F_2$ with $F_2\in \mathcal{M}$ and $F_2(0)=0$. Similarly by using the properties of $\mathcal{M}$ we conclude that $S^*F_2=G_2+\beta_2E_1$ for some $G_2\in \mathcal{M}$ and $\beta_2\in \mathbb{C}$ which again implies that $ F_2=zG_2+\beta_2zE_1$.
	Therefore in the second iteration we have
	\begin{equation*}
	F(z)=F_0(z)(A_0+A_1z)+z^2G_2(z)+(\beta_1z+\beta_2z^2)E_1(z) \quad \quad (z\in \mathbb{D})
	\end{equation*}
	and $\|F\|^2 = \norm{A_0}^2+\norm{A_1}^2+ \norm{G_2}^2 + |\beta_1|^2 + |\beta_2|^2.$
	If we continue the above process at the $k$-th iteration we obtain,
	\begin{align}\label{app1}
	F(z)=F_0(z)(A_0+A_1z+\cdots+A_{k-1}z^{k-1})+z^kG_k(z)+(\beta_1z+\cdots+\beta_kz^k)E_1(z)
	\end{align}
	and
	\begin{align}\label{app2}
	\norm{F}^2=\sum_{j=0}^{k-1}\norm{A_j}^2+\norm{G_k}^2+\sum_{j=1}^{k}
	|\beta_j|^2.
	\end{align}
	Moreover from the above iterations we also note that  $G_k=P_1S^*P_2(G_{k-1})$, where $P_1$ and $P_2$ are the orthogonal projections with kernel $\langle E_1 \rangle$ and $\wcl$ respectively. Since $S\in C_{\cdot 0}$, ~$\dim{\mathcal{D}_S<\infty}$ and $P_1$ is an orthogonal projection with finite dimensional kernel, then by applying lemma \ref{abc} we  conclude that $SP_1\in C_{\cdot 0}$. Furthermore, since ~$\dim{\mathcal{D}_{SP_1}<\infty}$ and $P_2$ is an orthogonal projection with finite dimensional kernel, then by applying lemma \ref{abc} once again we  conclude that $SP_1P_2\in C_{\cdot 0}$. Again from the above iterations we note that 
	\begin{align*}
	G_k&=(P_1S^*P_2)^k(F) =P_1S^*(P_2P_1S^*)^{k-1}P_2(F)=P_1S^*(SP_1P_2)^{*k-1}P_2(F)
	\end{align*}
	and hence $\norm{G_k}\leq \norm{(SP_1P_2)^{*k-1}P_2(F)}\to 0, ~\text{as}~ k\to \infty.$
	Consequently from the above equations \eqref{app1} and \eqref{app2} we can write
	\begin{equation} \label{app3}
	F(z)=F_0(z)K_0(z)+zk_1(z) E_1(z)\quad \text{and} \quad \norm{F}^2=\norm{K_0}^2+\norm{k_1}^2,
	\end{equation}
	where
	$$K_0(z)=\sum_{j=0}^{\infty}A_jz^j, ~k_1(z)=\sum_{j=1}^{\infty}\beta_jz^{j-1} $$
	and the sums converge in $\hdcr$ and $\hdcc$ norm respectively. Thus finally we say that if $F\in \mcl$ then 
	\begin{equation*}
	F=F_0K_0+zk_1E_1,
	\end{equation*}
	where $(K_0,k_1)\in\hdcr\times\hdcc$. Recall that $\hdcr\times\hdcc$ can be identified with $\hdcrone$. Define the subspace $\kcl$ of $\hdcrone$ as follows:
	\begin{align*}
	&\kcl=\Big\{ (K_0,k_1)\in\hdcrone:\exists~ F\in \mcl \text{~such~that~}\\
	&\hspace{3.5in} F=F_0K_0+zk_1E_1\Big\}.
	\end{align*}
	Then by using \eqref{app3} we conclude that $\kcl$ is a closed subspace of $\hdcrone$. Now it remains to show that $\kcl$ is $S^*\oplus\cdots\oplus S^*$-invariant  in $\hdcrone$. Indeed, let $(K_0,k_1)\in\kcl$. Then there exists $F$ in $\mcl$ such that $F=F_0K_0+zk_1E_1$. On the other hand
	\begin{align*}
	F&=F_0K_0+zk_1E_1 =F_0A_0+F_0(K_0-K_0(0))+zk_1E_1\\
	&=F_0A_0+\{F_0(K_0-K_0(0))+z(k_1-k_1(0))E_1+zk_1(0)E_1\}\\
	&=F_0A_0+\underbrace{z(F_0S^*K_0+zS^*k_1E_1)}+z\beta_1E_1,
	\end{align*}
	which along with equation \eqref{app4} implies $F_0S^*K_0+zS^*k_1E_1=G_1\in \mathcal{M}$, which proves that $\kcl$ is $\underbrace{S^*\oplus\cdots\oplus S^*}_{r+1}$-invariant.
	
	\noindent Conversely, let $\mcl$ be a closed subspace of $\hdcm$ which has a representation like \eqref{M}.
	Let $F\in\mcl$ be such that $F(0)=0$. Then there exists $(K_0,k_1)$ in $\kcl$ such that $F=F_0K_0+zk_1E_1$. Now it is easy to observe that $\big\{W_i(0)\big\}_{i=1}^r$ is linearly independent which follows from the fact that $\wcl$ is a linear space and all functions in $\wcl$ do not vanish at $0$.
	%Next we claim that $\big\{W_i(0)\big\}_{i=1}^r$ is linearly independent. On contrary let us assume that $\big\{W_i(0)\big\}_{i=1}^r$ is linearly dependent, then there exist scalars $c_1,\ldots,c_r$ not all zero such that $$c_1W_1(0)+c_2W_2(0)+\cdots+c_rW_r(0)=0.$$ As $W_1,W_2,...,W_r\in\wcl$ then $c_1W_1+c_2W_2+\cdots+c_rW_r\in\wcl$ but functions in $\wcl$ do not vanish at $0$. Which is a contradiction.\\
	On the other hand $F(0)=0$ and the fact $\big\{W_i(0)\big\}_{i=1}^r$ is linearly independent together implies $K_0(0)=0$. 
	%Let $K_0(0)=(\alpha_1,\alpha_2,...,\alpha_r)\in\mathbb{C}^r$. Then $F(0)=0$ implies that $F_0(0)K_0(0)=0$. Since $F_0(0)K_0(0)=\alpha_1W_1(0)+\cdots+\alpha_rW_r(0)$ and $\big(W_i(0)\big)_{i=1}^r$ is linearly independent, $K_0(0)=0$. Consider
	Thus 
	\begin{align*}
	S^*F(z)&=\dfrac{F(z)-F(0)}{z}=\dfrac{F_0(z)K_0(z)+zk_1(z)E_1(z)-F_0(0)K_0(0)}{z}\\
	&=F_0(z)S^*K_0(z)+(SS^*k_1(z))E_1(z)+k_1(0)E_1(z)
	\end{align*}
	and hence $S^*F=F_0S^*K_0+zS^*k_1E_1+k_1(0)E_1$. On the other hand since $\kcl$ is $S^*\oplus\cdots\oplus S^*$-invariant in $\hdcrone$, then we have $(S^*K_0,S^*k_1)\in\kcl$ and hence $F_0S^*K_0+zS^*k_1E_1\in\mcl$. This shows that $S^*F\in\mcl\oplus\fcl$ and consequently $\mcl$ is nearly $S^*$- invariant with defect one.
	
	%Since $\kcl$ is $S^*$-invariant,$(S^*K_0,S^*k_1)\in\kcl$ and due to the representation of $\mcl$, $F_0(S^*K_0)+z(S^*k_1)E_1\in\mcl$.\\
	%Hence $S^*F\in\mcl\oplus\fcl$. Consequently, $\mcl$ is a nearly $S^*$ invariant subspace with defect one.
	\vspace*{0.1in}
	
	(ii) If $\mcl\subseteq z\hdcm$, then $\wcl=\{0\}$. Therefore by applying the similar kind of algorithm as in $(i)$, $F$ can be written as $F(z)=zk_1(z)E_1(z)$ for some $\hdcc$ fucntion $k_1$. In other words $\mcl$ has the following representation. $$\mcl= \{F : F(z)= zk_1(z)E_1(z) : k_1\in \kcl\}, $$ where $\kcl$ is a closed $S^*$- invariant subspace of the Hardy space $\hdcc$ and $\norm{F}^2=\norm{k_1}^2 .$ This completes the proof.
\end{proof}

In general for finite defect $p$ the analogous calculations produce the following result.

\begin{thm}\label{ab} \footnote[1]{Recently this result also obtained independently by Ryan O\textquoteright Loughlin in \cite{OR}.}
	Let $\mcl$ be a closed subspace that is nearly $S^*$-invariant with defect $p$ in $\hdcm$ and let $\{E_1,E_2,.....E_p \}$ be any orthonormal basis for the $p$-dimensional defect space $\fcl$. Let $\{W_1,W_2,\ldots,W_r\}$ be an orthonormal basis of $\wcl:=\mcl\ominus(\mcl\cap z\hdcm)$ and let $F_0$ be the $m\times r$ matrix whose columns are $W_1,W_2,\ldots,W_r$. Then\vspace*{0.1in}\\
	(i) in the case where there are functions in $\mcl$ that do not vanish at $0$,
	\begin{equation}\label{main1}
	\mcl= \Big\{F : F(z)= F_0(z)K_0(z)+ \sum_{j=1}^{p} zk_j(z)E_j(z) : (K_0,k_1,\ldots,k_p)\in \kcl \Big\},
	\end{equation}
	where 
	$\kcl \subset \hdcr \times \underbrace{\hdcc\times\cdots\times \hdcc}_{p} $ is a closed $S^*\oplus\cdots\oplus S^*$- invariant subspace of the vector valued Hardy space $ H ^2_{\mathbb{C}^{r+p}}(\D)$ and $$\norm{F}^2=\norm{K_0}^2+\sum_{j=1}^{p}\norm{k_j}^2 .$$
	
	\noindent (ii) In the case where all the functions in $\mcl $ vanish at $0$,
	\begin{equation}\label{main2}
	\mcl= \Big\{F : F(z)= \sum_{j=1}^{p} zk_j(z)E_j(z) : (k_1,\ldots,k_p)\in \mathcal{K} \Big\},
	\end{equation}
	with the same notion as in $(i)$ except that $\mathcal{K}$ is now a closed 	$S^*\oplus \cdots\oplus S^*$- invariant subspace of the vector valued Hardy space $ H ^2_{\mathbb{C}^{p}}(\D)$ and $$\norm{F}^2=\sum_{j=1}^{p}\norm{k_j}^2 .$$
	Conversely, if a closed subspace $\mcl$ of the vector valued Hardy space $\hdcm$ has a representation like $(i)$ or $(ii)$ as above, then it is a nearly $S^*$-invariant subspace of defect $p$. 
\end{thm}

\begin{crlre}
	A closed subspace $\mcl \subset \hdcm$ is an almost invariant subspace for $S^*$ with defect p if and only if it satisfies the conditions of the above Theorem \ref{ab} together with an extra condition that $S^* W_i \in \mcl \oplus \fcl$ for all $i=1,2,\ldots,r$ in case $(i)$, while case $(ii)$ is unchanged.
\end{crlre}
\begin{proof}
	If $\mcl$  is an almost invariant subspace for $S^*$ with defect $p$, then it is nearly $S^*$-invariant subspace with defect $p$. Thus it satisfies the conditions of the above Theorem \ref{ab} and since $W_1, W_2,\ldots,W_r\in \mcl$, then from the hypothesis it follows that $S^* W_i \in \mcl \oplus \fcl$ for all $i=1,2,\ldots,r$.
	Conversely, assume that $\mcl $ satisfies the conditions of Theorem \ref{ab} together with $S^*W_i \in \mcl \oplus \fcl$ for all $i=1,2,\ldots,r$.
	Thus for any $F \in \mcl$ we have
	$$F=F_0K_0+ \sum_{j=1}^{p}zk_jE_j=F_0K_0(0)+K,$$ where $$K(z)=F_0(z)(K_0(z)-K_0(0))+ \sum_{j=1}^{p}zk_j(z)E_j(z).$$
	Observe that $K(0)=0$ and $F_0K_0(0)\in \mcl$ implies $K \in \mcl$. Since $\mcl$ is nearly $S^*$ invariant with defect $p$, then $S^*K\in \mcl \oplus \fcl$ which along with the fact $S^*W_i\in \mcl \oplus \fcl$ for all $i=1,\ldots,r$ implies $S^*F \in \mcl \oplus \fcl$ and hence $\mcl$ is an almost invariant with defect $p$. This completes the proof.
\end{proof}

\begin{rmrk}
	It is easy to observe that $S^*\mcl \subset \mcl \oplus \fcl$ is equivalent to the condition that $S(\mcl \oplus \fcl)^\perp \subset (\mcl \oplus \fcl)^\perp \oplus \fcl$.
	Therefore almost invariant subspaces for $S$ can be characterized by $S^*$- invariant subspaces.
\end{rmrk}

\section{Characterization of nearly invariant subspaces in terms of shift invariant subspaces }
In this section our main aim is to give a connection between nearly $S^*$-invariant  subspaces and $S$-invariant subspaces using our main result in the previous section.
%\begin{dfn}
%	Let $\Theta:\D\to\mathbb{C}^{m\times n}$ such that there is a power series expansion $\Theta(z)=\sum_{k=0}^{\infty}z^k\Theta_k$ with $\Theta_k\in\mathbb{C}^{m\times n}$ that is converge on $\D$.
%\end{dfn}
In other words we try to characterize $\mcl ^\perp$ in terms of shift invariant subspaces, where $\mcl$ is a nearly $S^*$- invariant subspace of $\hdcm$ with finite defect $p$. Note that in general it is difficult to deal with the case when the defect has an orthonormal basis consist of arbitrary  functions of $\hdcm$. Therefore we restrict ourselves in the special case when the orthonormal basis for the defect space are bounded analytic functions that is, $\fcl = span \{E_1,E_2,\ldots,E_p\}$, where $E_i \in   H ^{\infty}_{\mathcal{L}(\C,\C ^m)}(\D)$ for $i=1,\ldots,p$.
Let $F_0\in  H ^{2}_{\mathcal{L}(\C^r,\C ^m)}(\D)$ and $T_{F_0}:\hdcr\to\hdcm$ be an operator defined by $T_{F_0}(G)=P(F_0G)$, where $P$ is the Fourier projection of the $L^1(\mathbb{T},\mathbb{C}^m)$ function $F_0G$ on $ H^2_{\mathbb{C}^m}(\D)$.
%Thus $T_{E_j}$ is a matrix valued $(m\times 1)$ multiplier operator.
%For more details about about multiplier operator and bounded analytic functions see \cite{NUEM}.

First we consider the case when $\mcl$ is a nearly $S^*$- invariant subspace of $\hdcm$ which contain functions that do not vanish at $0$ with an extra assumption that $\wcl:=\mcl\ominus(\mcl\cap z\hdcm)$ has an orthonormal basis consist of $ H ^\infty_{\mathbb{C}^m}(\D)$-functions, that is, $\{W_1,\ldots, W_r\}\subseteq   H ^\infty_{\mathbb{C}^m}(\D)$. Therefore $\mathcal{M}$ has the form \eqref{main1} by Theorem \ref{ab} $(i)$.
Now consider $G\in \mcl^{\perp}$, then 
$\langle G,F\rangle =0 ,\forall F \in \mcl$. But for any $F\in \mcl$ we have $$F= F_0K_0 +\sum_{j=1}^{p}Sk_jE_j,$$ where $(K_0,k_1,\ldots,k_p)\in \kcl$ and $\kcl$ is a $S^*$- invariant subspace of $  H ^2_{\mathbb{C}^{r+p}}(\D).$
Therefore
\begin{align*}
\langle G,F\rangle&=\langle G,F_0K_0\rangle_2 +\sum_{j=1}^p\langle G,Sk_jE_j \rangle_2  =\langle G,T_{F_0}K_0\rangle_2 +\sum_{j=1}^p\langle S^*G,k_jE_j \rangle_2\\
&=\langle T^*_{F_0}G,K_0\rangle_{\hdcr} +\sum_{j=1}^p\langle T^*_{E_j}S^*G,k_j \rangle_{\hdcc},
\end{align*}
which ensures $$G\in \mcl^{\perp} \textit{~if~and~only~if~} T^*_{F_0}G \oplus T^*_{E_1}S^*G\oplus T^*_{E_2}S^*G\oplus \cdots\oplus T^*_{E_p}S^*G \in \kcl^{\perp}.$$
%and $\kcl^{\perp}$ is a $S$- invariant subspace of $ H ^2_{\mathbb{C}^{r+p}}(\D)$.
%On the other hand if $G\in \hdcm$ such that $$T^*_{F_0}G \oplus T^*_{E_1}S^*G\oplus T^*_{E_2}S^*G\oplus \cdots\oplus T^*_{E_p}S^*G \in \kcl^{\perp},$$ then 
%\begin{align*}
%&\langle T^*_{F_0}G,K_0\rangle_{\hdcr} +\langle T^*_{E_1}S^*G,k_1 \rangle_{\hdcc} +\langle T^*_{E_2}S^*G,k_2 \rangle_{\hdcc} +\\&
%\cdots +\langle T^*_{E_p}S^*G,k_p \rangle_{\hdcc}=0, ~~~\forall (K_0,k_1,...,k_p)\in \kcl\\
%&\Rightarrow \langle G,F\rangle =0,~~~\forall F \in\mcl.
%\end{align*}
Thus 
\begin{equation}\label{main3}
\mcl ^{\perp}=\Big\{G\in \hdcm :T^*_{F_0}G \oplus T^*_{E_1}S^*G\oplus T^*_{E_2}S^*G\oplus \cdots\oplus T^*_{E_p}S^*G \in \kcl^{\perp}\Big \},   
\end{equation}
where $\kcl ^{\perp}$ is a $S$- invariant subspace of $ H ^2_{\mathbb{C}^{r+p}}(\D).$ Conversely, if $\mcl$ is a closed subspace of $\hdcm$ such that $\mcl ^\perp$ is of the form \eqref{main3}, then  $\mcl$ is a nearly $S^*$- invariant subspace of $\hdcm$ with defect $p$.
%$\mcl ^\perp$ contain functions that do not vanish at $0$ and $F_0$ is the $m\times r$ matrix whose columns are $W_1,\ldots,W_r \in\mathbf{H}^\infty(\D,\mathbb{C}^m)$ which forms an orthonormal basis for $\wcl := \mcl \ominus (\mcl \cap z\hdcm)$ and $E_1,...,E_p$ are in $\mathbf{H}_{m\times 1}^{\infty}(\D)$ and also if $$\mcl= \{G\in \hdcm :T^*_{F_0}G \oplus T^*_{E_1}S^*G\oplus T^*_{E_2}S^*G\oplus \cdots\oplus T^*_{E_p}S^*G \in \kcl^{\perp} \}$$	then $\mcl ^{\perp}$ is nearly $S^*$ invariant subspace of $\hdcm$ of defect $p$.\\
%The proof of this part is trivial by theorem \ref{ab}.\\
Similarly we can also obtain the expression of $\mcl ^\perp$ in the case  when $\mcl \subset z\hdcm$. We can then formulate our main result in this section as follows.
\begin{thm}\label{abcd}
	Let $\mcl$ be a closed subspace of  $\hdcm$ which is nearly $S^*$- invariant with defect $p$ with an extra condition that the orthonormal basis for both the defect space and the space $\mcl\ominus(\mcl\cap z\hdcm)$  are bounded analytic functions. Then,\vspace*{0.1in}\\
	(i) in the case where there are functions in $\mcl$ that do not vanish at $0$,
	$$\mcl ^\perp =\Big\{G\in \hdcm :T^*_{F_0}G \oplus T^*_{E_1}S^*G\oplus T^*_{E_2}S^*G\oplus \cdots\oplus T^*_{E_p}S^*G \in \kcl^{\perp} \Big\},$$
	where  $\{W_i\}_{i=1}^{r}$ is an orthonormal basis of $\wcl := \mcl \ominus (\mcl \cap z\hdcm)$,  $F_0$ is the $m\times r$ matrix whose columns are $W_1,W_2,\ldots,W_r$, the defect space $\fcl$ has an orthonormal basis $\{E_1,E_2,\ldots,E_p \}\subseteq  H ^{\infty}_{ \mathcal{L}(\mathbb{C},\mathbb{C}^m)}(\D)$  and $\kcl ^\perp \subseteq \hdcr \times \underbrace{\hdcc\times \cdots\times \hdcc}_{p} $ is a closed $S\oplus \cdots\oplus S$- invariant subspace of the vector valued Hardy space $ H ^2_{\mathbb{C}^{r+p}}(\D)$.
	\vspace*{0.1in}\\
	(ii)\hspace*{0.1in} In the case where all functions in $\mcl$	vanish at 0, 
	$$\mcl ^\perp =\Big\{G\in \hdcm : T^*_{E_1}S^*G\oplus T^*_{E_2}S^*G\oplus \cdots\oplus T^*_{E_p}S^*G \in \kcl^{\perp} \Big\},$$
	where $\{E_1,E_2,.....E_p \}\subseteq H ^{\infty}_{ \mathcal{L}(\mathbb{C},\mathbb{C}^m)}(\D)$ is an orthonormal basis for the $p$ dimensional defect space $\fcl$ and $\kcl ^\perp \subseteq  \underbrace{\hdcc\times \cdots\times \hdcc}_{p} $ is a closed $S\oplus\cdots\oplus S$- invariant subspace of the vector valued Hardy space $ H ^2_{\mathbb{C}^{p}}(\D)$. 
	Conversely, if a closed subspace $\mcl \subset \hdcm$ is such that $\mcl ^\perp$ has a representation as in (i) or (ii), then $\mcl$ is a nearly $S^*$- invariant subspace of defect p.	
\end{thm}		

Finally we end the section with the following remark. 
\begin{rmrk}
	In this section we describe nearly invariant subspaces under the backward shift with finite defect with an extra assumption that bounded analytic functions form an orthonormal basis of the defect space. But we do expect that the version of Theorem \ref{abcd} still hold without this assumption.
	
\end{rmrk}

\section*{Acknowledgements}
\textit{ We are extremely grateful to Dr. Bata Krishna Das for many fruitful discussions and his valuable comments. We would also like to thank Prof. Joydeb Sarkar for introducing this area to us.}

\end{document}